\documentclass[12pt]{amsart}

\begin{document}
\numberwithin{equation}{section}

\def\1#1{\overline{#1}}
\def\2#1{\widetilde{#1}}
\def\3#1{\widehat{#1}}
\def\4#1{\mathbb{#1}}
\def\5#1{\frak{#1}}
\def\6#1{{\mathcal{#1}}}

\def\C{{\4C}}
\def\R{{\4R}}
\def\N{{\4N}}
\def\Z{{\4Z}}

\title[One-resonant biholomorphisms]{Dynamics of one-resonant biholomorphisms}
\author[F. Bracci]{Filippo Bracci}
\address{F. Bracci: Dipartimento Di Matematica, Universit\`{a} Di Roma \textquotedblleft Tor
Vergata\textquotedblright, Via Della Ricerca Scientifica 1,
00133, Roma, Italy. } \email{fbracci@mat.uniroma2.it}
\author[D. Zaitsev]{Dmitri Zaitsev*}
\address{D. Zaitsev: School of Mathematics, Trinity College Dublin, Dublin 2, Ireland}
\email{zaitsev@maths.tcd.ie}
\thanks{$^{*}$Supported in part by the Science Foundation Ireland grant 06/RFP/MAT018.}


\def\Label#1{\label{#1}{\bf (#1)}~}


\def\cn{{\C^n}}
\def\cnn{{\C^{n'}}}
\def\ocn{\2{\C^n}}
\def\ocnn{\2{\C^{n'}}}


\def\dist{{\rm dist}}
\def\const{{\rm const}}
\def\rk{{\rm rank\,}}
\def\id{{\sf id}}
\def\aut{{\sf aut}}
\def\Aut{{\sf Aut}}
\def\CR{{\rm CR}}
\def\GL{{\sf GL}}
\def\Re{{\sf Re}\,}
\def\Im{{\sf Im}\,}
\def\span{\text{\rm span}}

\def\codim{{\rm codim}}
\def\crd{\dim_{{\rm CR}}}
\def\crc{{\rm codim_{CR}}}

\def\phi{\varphi}
\def\eps{\varepsilon}
\def\d{\partial}
\def\a{\alpha}
\def\b{\beta}
\def\g{\gamma}
\def\G{\Gamma}
\def\D{\Delta}
\def\Om{\Omega}
\def\k{\kappa}
\def\l{\lambda}
\def\L{\Lambda}
\def\z{{\bar z}}
\def\w{{\bar w}}
\def\t{\tau}
\def\th{\theta}
\def\ta{\tilde{\alpha}}
\def\sideremark#1{\ifvmode\leavevmode\fi\vadjust{
\vbox to0pt{\hbox to 0pt{\hskip\hsize\hskip1em
\vbox{\hsize1.5cm\tiny\raggedright\pretolerance10000
\noindent #1\hfill}\hss}\vbox to8pt{\vfil}\vss}}}

\def\Dif{{\sf Diff}(\C^n;0)}

\emergencystretch15pt \frenchspacing

\newtheorem{theorem}{Theorem}[section]
\newtheorem{lemma}[theorem]{Lemma}
\newtheorem{proposition}[theorem]{Proposition}
\newtheorem{corollary}[theorem]{Corollary}

\theoremstyle{definition}
\newtheorem{definition}[theorem]{Definition}
\newtheorem{example}[theorem]{Example}

\theoremstyle{remark}
\newtheorem{remark}[theorem]{Remark}
\numberwithin{equation}{section}

\begin{abstract} Our first main result is a construction of a simple
formal normal form for holomorphic
diffeomorphisms in $\C^n$ whose
differentials have one-dimensional family of resonances in
the first $m$ eigenvalues, $m\leq n$ (but more resonances are
allowed for other eigenvalues).
Next, we provide invariants and give conditions
for the existence of basins of attraction.
Finally, we give applications and examples
demonstrating the sharpness of our conditions.
\end{abstract}

\maketitle

\section{Introduction}

Let $F$ be a germ of holomorphic diffeomorphism of $\C^n$
fixing the origin $0$ with diagonalizable differential. The
dynamical behavior of the sequence of iterates $\{F^{\circ
q}\}_{q\in \N}$ of $F$ in a neighborhood of $0$ is depicted at
the first order by the dynamics of its differential $dF_0$. In
fact, depending on the eigenvalues $\lambda_1,\ldots,
\lambda_n$ of $dF_0$, in some cases both dynamics are the same.

In the hyperbolic case (namely when none of the eigenvalues is
of modulus $1$) the map is topologically conjugated to
its differential (by the Hartman-Grobman  theorem \cite{Har},
\cite{G1}, \cite{G2}) and the dynamics is clear. Moreover, if
the eigenvalues have either all modulus strictly smaller than one or
all strictly greater than one, then the origin is an attracting or respectively
repelling fixed point for an open neighborhood of $0$. Also, by
the stable/unstable manifold theorem, there exists a
holomorphic (germ of) manifold invariant under $F$ and tangent
to the sum of the eigenspaces of those $\lambda_j$'s such that
$|\lambda_j|<1$ ({\sl resp.} $|\lambda_j|>1$) which is
attracted to ({\sl resp.} repelled from)  $0$. However,
already in  case when all eigenvalues have modulus different
from $1$, holomorphic linearization is not always possible due
to the presence of resonances among the eigenvalues (see, for
instance, \cite[Chapter IV]{Ar}).

The case where some eigenvalue has modulus $1$ is the most
``chaotic'' and interesting, since it presents a plethora of
possible scenarios. For instance, if those eigenvalues of
modulus $1$ are not roots of unity and satisfy some Bruno-type
conditions, then there exist Siegel-type invariant submanifolds
(see \cite{Po}, \cite{Y}) on which the map is (holomorphically)
linearizable. If the map is tangent to the identity,
it has been proved by \'Ecalle \cite{Ec} and Hakim \cite{Ha}
that generically there exist ``petals'', also called ``parabolic
curves'', namely, one-dimensional $F$-invariant analytic discs
having the origin in their boundary and on which the dynamics
is of parabolic type. Later, Abate \cite{Ab} (see also \cite{ABT}) proved that such
petals always exist in dimension two.

On the other hand, Hakim~\cite{H} (based on the previous work
by Fatou \cite{Fa} and Ueda \cite{Ue1}, \cite{Ue2} in $\C^{2}$,
see also Takano \cite{Taka})  studied the so-called {\em
semi-attractive} case, with one eigenvalue equal to $1$ and the
rest of eigenvalues having modulus less than $1$. She proved
that either there exists a curve of fixed points or there exist
attracting open petals. Such a result has been later
generalized by Rivi \cite{Ri}.

The {\em quasi-parabolic} case of a germ in $\C^2$, i.e.\ having one
eigenvalue $1$ and the other of modulus equal to one, but not a
root of unity has been
studied in \cite{B-M} and it has been proved that, under a certain
generic hypothesis called ``dynamical separation'', there
exist petals tangent to the eigenspace of $1$. Such a result
has been generalized to higher dimension by Rong \cite{R1},
\cite{R2}. We refer the reader to the survey papers \cite{Ab1}
and \cite{Br} for a more accurate review of existing results.

In case of diffeomorphisms with {\em unipotent} linear part, it
was shown by Takens \cite{Tak} (see also \cite[Chapter 1]{IY})
that such diffeomorphism can be embedded in the flow of a
formal vector field. Therefore, in this case the dynamics of
the diffeomorphism, at least at the formal level, is related to
that of a (formal) associated vector field. For instance, using
the Camacho-Sad theorem on the existence of separatrices for
vector fields \cite{CS}, Brochero, Cano and Hernanz \cite{BCL}
gave another proof of Abate's theorem. On the other hand, when
the linear part of the diffeomorphism is not unipotent, the
authors are not aware of any general result about embedding
such a diffeomorphism into the flow of a formal vector field.
In fact, one encounteres somewhat unexpected differences
between the dynamics of diffeomorphisms and that of vector
fields, see Raissy \cite{Ra}.

The aim of the present paper is the study of normal forms and
the dynamics of germs of holomorphic diffeomorphisms having a
one-dimensional family of resonances among only certain
eigenvalues (that we call here {\em partially one-resonant}
diffeomorphisms). It should be mentioned here that (fully)
one-resonant vector fields have been studied by Stolovitch in
\cite{St}, where he also obtained a normal form for vector
fields up to multiplication by a unit. In case of
diffeomorphisms considered here, there is no natural analogue
of multiplying by a unit and thus we are lead to seek a normal
form for the original diffeomorphism only under conjugations.

More in details, let $\lambda_1,\ldots, \lambda_n$ be the eigenvalues
of the linear part of a biholomorphic diffeomorphism germ $F$ at $0$.
We say that $F$ is
{\sl one-resonant with respect to the first $m$ eigenvalues
$\{\lambda_1,\ldots, \lambda_m\}$} ($1\le m\leq n$) (or {\sl partially one-resonant}) if there
exists a fixed multi-index $\alpha=(\alpha_1,\ldots, \alpha_m,
0,\ldots, 0)\ne0\in\N^n$  such for $s\le m$,
the resonances $\l_{s}=\prod_{j=1}^n\lambda_j^{\b_j}$
 are precisely of the form $\l_{s}=\l_{s}\prod_{j=1}^m\lambda_j^{k\a_j}$,
where $k\ge1\in\N$ is arbitrary.
We stress out that, since arbitrary resonances are allowed for $s>m$, such a
condition is much weaker (see Example~\ref{mn}) than the one-resonance condition
normally found in the literature corresponding here to the case
$m=n$, see e.g.\ \cite{I,St}. The main advantage of the new
notion of partial one-resonance is that it can be applied to
the subset of all eigenvalues of modulus equal to $1$ that is
natural to  treat differently from the rest of the eigenvalues.

In case of partial one-resonance, the classical
Poincar\'e-Dulac theory implies that, whenever $F$ is not
formally linearizable in the first $m$ components, $F$ is
formally conjugated to a map whose first $m$ components are of
the form $\lambda_j z_j +a_j z^{\alpha k}z_j + R_j(z)$,
$j=1,\ldots, m$, where, the number $k\in \N$ is an invariant,
called {\sl the order of $F$ with respect to
$\{\lambda_1,\ldots, \lambda_m\}$}, the vector $(a_1,\ldots,
a_m)\neq 0$ is invariant up to a scalar multiple and the
$R_j$'s contain  only resonant terms of higher degree. The
number
\[
\Lambda=\L(F):=\sum_{j=1}^m \frac{{a_j} \alpha_j}{{\lambda_j}}
\]
is an invariant up to a scalar multiple, and the map $F$ is
said to be {\sl non-degenerate} provided $\Lambda\neq 0$.

We show that (partially) one-resonant non-degenerate diffeomorphisms have a simple
formal {\em normal form} (see Theorem \ref{normal}) in which the
first $m$ components are of the form
\begin{equation*}
\lambda_j z_j +a_j z^{k\a}z_j +\mu {\alpha_j}{\1\l_{j}}^{-1}z^{2k\a} z_j,
\quad j=1,\ldots, m.
\end{equation*}
Although none of the eigenvalues $\lambda_j$, $j=1,\ldots, m$,
might be roots of unity, such a normal form is the exact
analogue of the formal normal form for parabolic germs in $\C$.
In fact, a one-resonant germ acts as a parabolic germ on the
space of leaves of the formal invariant foliation $\{z^\a=\hbox{const}\}$
and that is  the reason for this parabolic-like behavior.

Let $F$ be a one-resonant non-degenerate diffeomorphism with
respect to the eigenvalues $\{\lambda_1,\ldots, \lambda_m\}$.
We say that $F$ is {\sl parabolically attracting} with respect
to $\{\lambda_1,\ldots, \lambda_m\}$ if
\begin{equation*}
|\l_{j}|=1, \quad
\Re\left( {a_j}{\lambda_j}^{{-1}} {\Lambda}^{-1}\right)>0, \quad j=1,\ldots, m.
\end{equation*}
Again, such a condition is invariant and its inequality part is vacuous in
dimension $1$ or whenever $m=1$ and $|\lambda_1|=1$ since in
that case $\a=(\a_{1},0,\ldots,0)$ with $\a_{1}>0$. Our main
result is the following:

\begin{theorem}\label{main}
Let $F$ be a holomorphic diffeomorphism germ at $0$ that is
 one-resonant, non-degenerate and
parabolically attracting  with respect to
$\{\lambda_1,\ldots,\lambda_m\}$. Suppose that $|\lambda_j|<1$
for $j>m$. Let $k\in \N$ be the order of $F$ with respect to
$\{\lambda_1,\ldots, \lambda_m\}$. Then $F$ has $k$ disjoint
basins of attraction having $0$ on the boundary.
\end{theorem}

The different basins of attraction for $F$ (that may or may not
be connected) project via the map $z\mapsto u=z^\a$  into
different petals of the germ $u\mapsto
u+\Lambda(F)u^{k+1}+o(|u|^{k+1})$.

Theorem~\ref{main} has many consequences. For instance,
we recover a result of Hakim (see Corollary
\ref{hakim}) since
not formally linearizable semi-attractive  germs are
always one-resonant, non degenerate and parabolically
attracting. Also, we apply our machinery to the case of
quasi-parabolic germs, providing ``fat petals'' in the
quasi-parabolic dynamically separating and attracting cases (see
Subsection \ref{nonsep}). Another area of application of
Theorem \ref{main} concerns elliptic germs which, in dimension
greater than $1$, might present some, maybe unexpected,
parabolic-like behavior, see Subsection~\ref{mixed}.
Finally, we present examples of a one-resonant degenerate
as well as non-degenerate but not parabolically attracting
germs  which have no basins of attraction at $0$,
demonstrating sharpness of the assumptions of Theorem~\ref{main},
see Subsections~\ref{nobase}
and \ref{non-attracting}.

The outline of the paper is as follows. In Section~\ref{2} we
 briefly recall the one-dimensional theory of parabolic germs and define
one-resonant germs in higher dimension. In Section~\ref{3} we
construct a formal normal form for non-degenerate partially
one-resonant germs. In Section \ref{five} we study the dynamics
of normal forms, as a motivation for the subsequent
Section~\ref{prova}, where we give the proof of Theorem
\ref{main}. Finally, in Section \ref{appl} we apply our theory to the
semi-attractive case, quasi-parabolic case, elliptic case and
provide examples of diffeomorphism with no basins of attraction.

\medskip

The first named author wishes to thank Jasmin Raissy for some
helpful conversations. Also, both authors thank the referee for
his/her very useful comments.

\section{One-resonant diffeomorphisms}\label{2}

\subsection{Preliminaries on germs tangent to the identity in
$\C$}\label{preli} (see  e.g.\ \cite{CaGa}). Let
\begin{equation}\label{h}
h(u):=u+Au^{k+1}+O(|u|^{k+2})
\end{equation}
 for some $A\neq 0$ and $k\geq 1$,
be a germ at $0$ of a holomorphic self-mapping of $\C$.

The {\sl attracting directions} $\{v_1,\ldots, v_k\}$ for $h$
are given by the $k$-th roots of
$-\frac{|A|}{A}$.
These are precisely the directions $v$ such that the term $Av^{k+1}$
shows in the direction opposite to $v$.
An {\sl attracting petal} $P$ for $h$ is a simply-connected
domain such that $0\in \partial P$, $h(P)\subseteq P$ and
$\lim_{m\to\infty}h^{\circ m}(z)=0$ for all $z\in P$,
where $h^{\circ m}$ denotes the $m$th iterate of $h$.

We state here (a part of) the Leau-Fatou flower theorem. We
write $a\sim b$ whenever there exists constants $0<c<C$ such
that $ca\le b\le Ca$.

\begin{theorem}[Leau-Fatou]\label{LF}
Let $h(u)$ be as in \eqref{h} and $v$
an attracting direction for $h$ at $0$. Then there exists an attracting
petal $P$ for $h$ (said {\sl centered at $v$}) such that for each $z\in P$
the following hold:
\begin{enumerate}
  \item $h^{\circ m}(z)\ne0$ for all $m$ and $\lim_{m\to\infty}\frac{h^{\circ m}(z)}{|h^{\circ
  m}(z)|}=v$,
  \item $|h^{\circ m}(z)|^k\sim \frac{1}{m}$.
\end{enumerate}
Moreover, the petals centered at the attracting direction $v$
can be chosen to be connected components of the set
\[
\{z\in \C: |Az^k+\delta|<\delta\},
\]
where $0<\delta<<1$.
\end{theorem}

By the property (1), petals centered at different attracting
directions must be disjoint.

\begin{remark} Property (1) of Theorem \ref{LF} is a part of the standard statement
of the Leau-Fatou theorem (see, {\sl e.g.}, \cite{Ab1} or
\cite{Br2}). Property (2) follows from construction of the
so-called Leau-Fatou coordinate. We sketch it briefly here for
the reader convenience. Up to a dilation one can assume
$A=-1/k$ and $v=1$. Let $H:=\{w \in \C: \Re w>0, \, |w|>C\}$
and $\Psi(w):=w^{-1/k}$ for $w\in H$ with the $k$-th root
chosen so that $1^{1/k}=1$. By the Leau-Fatou construction
(see, {\sl e.g.} \cite[pp.19-22]{Br2}) if $C>0$ is sufficiently
large then the set $P:=\Psi(H)$ is $h$-invariant and  the  map
$\phi:=\Psi^{-1}\circ h\circ \Psi\colon H\to H$ satisfies
\begin{equation*}
\phi(w)=w+1+O(|w|^{-1}), \quad w\in H.
\end{equation*}
From here both (1) and (2) follow easily.
\end{remark}

\subsection{Partially one-resonant germs}

Let $\Dif$ denote the space of germs of holomorphic
diffeomorphisms of $\C^n$ fixing $0$. We shall adopt the
notation $\N=\{0,1,\ldots\}$. Given $\{\lambda_1,\ldots,
\lambda_n\}$ a set of complex numbers, recall that a {\em
resonance} is a pair $(j,l)$, where $j\in\{1,\ldots, n\}$ and
$l=(l_1,\ldots, l_n)\in\N^n$ is a multi-index with $|l|\ge2$
such that $\lambda_j=\lambda^l$ (where
$\lambda^l:=\lambda_1^{l_1}\cdots \lambda_n^{l_n}$).

In all the rest of the paper, and without mentioning it
explicitly, we shall consider only germs of diffeomorphisms
whose differential is diagonal.

\begin{definition}\label{one-resonant}
For $F\in \Dif$, assume that the differential $dF_0$  has
eigenvalues $\lambda_1,\ldots, \lambda_n$. We say that $F$ is
{\sl one-resonant with respect to the first $m$ eigenvalues
$\{\lambda_1,\ldots, \lambda_m\}$} ($1\le m\leq n$) if there
exists a fixed multi-index $\alpha=(\alpha_1,\ldots, \alpha_m,
0,\ldots, 0)\ne0\in\N^n$  such that the resonances $(j,l)$ with
$j\in\{1,\ldots,m\}$ are precisely of the form $(j,\a k+e_j)$,
where $e_j\in\N^n$ is the unit vector with $1$ at the $j$th
place and $0$ otherwise and where $k\ge1\in\N$ is arbitrary.
(In particular, it follows that the relation
$\l_1^{\a_1}\cdots\l_m^{\a_m}=1$ holds and generates
all other relations $\l_1^{\b_1}\cdots\l_n^{\b_n}=1$ with $\b_{s}\ge0$ for all $s$.) The multi-index $\alpha$ is
called the {\sl index of resonance}. If $F$ is one-resonant
with respect to $\{\lambda_1,\ldots, \lambda_n\}$ (i.e. $m=n$)
we simply say that $F$ is one-resonant.
\end{definition}

The notion of one-resonance for $m=n$ has been known in the literature,
see e.g.\ \cite{I,St}. However, its generalization for $m<n$ given here seems to be new.
The following class of examples illustrates the difference.

\begin{example}\label{mn}
Let $F\in{{\sf Diff}(\C^3;0)}$ be any diffeomorphism with
eigenvalues $\l,\mu,\nu$ of $dF_0$ such that $\l$ is a root of
unity, $|\mu|<1$ and $\nu=\mu^s$ for some natural number $s\ge
1$. Then $F$ is one-resonant with respect to $\l$ but has
resonances of the form $(3,se_2)$ showing it is not
one-resonant with respect to all the eigenvalues.
\end{example}

\begin{remark}\label{diff}
It follows directly from the definition that, if $F$ is
one-resonant with respect to $\{\lambda_1,\ldots, \lambda_m\}$,
then $\lambda_j\neq \lambda_s$ for any $j\in\{1,\ldots,m\}$ and $s\in\{1,\ldots,n\}$
with $j\neq s$. Indeed, otherwise one would have resonances of type
$(j,k\a+e_s)$ which are not of the required form $(j,k\a+e_j)$.
\end{remark}

\begin{example}\label{one-qp}
The same diffeomorphism can be considered one-resonant with
respect to different groups of eigenvalues. For instance,
consider $F(z,w)=(z+z^3, e^{2\pi i \theta}w+zw)$, where
$\theta$ is irrational. Then $F$ is one-resonant with respect
to $\lambda_1=1$ with index of resonance $(1,0)$. But also $F$
is one-resonant (with respect to $\{\l_1,\l_2\}=\{1,e^{2\pi i
\theta}\}$ with the same index of resonance $(1,0)$).
(Note that the higher order terms of F play no role here but will be
used later in Example~\ref{3.4}.)
\end{example}

As the previous example shows, there may exist ``non-maximal''
sets of ``one-resonant eigenvalues''. However, it is easy to
see from the definition that any set of ``one-resonant
eigenvalues'' is contained in the unique maximal set and
containes the unique minimal set. Namely, let $F$ be
one-resonant with respect to $\{\lambda_1,\ldots, \lambda_m\}$
and assume that the index of resonance is
$\a=(\a_1,\ldots,\a_m,0,\ldots,0)$. Since  the relation
$\l_1^{\a_1}\cdots\l_m^{\a_m}=1$ holds and generates all other
relations $\l_1^{\b_1}\cdots\l_n^{\b_n}=1$ with $\b_{s}\ge0$
for all $s$, it follows that any other resonant set of
eigenvalues corresponds to the same index $\a$. Then it follows
directly from the definition that every set of one-resonant
eigenvalues contains the minimal set $L$ of all $\l_j$ with
$\a_j\ne0$ and the set $L$ itself is one-resonant. On the other
hand, let $\2L$ be the set of all $\l_j$ such that any
resonance $(j,l)$ is of the required form $(j,k\a+e_j)$. Then
$\2L$ is the maximal one-resonant set that contains any other
one-resonant set of eigenvalues.

The choice of the set of eigenvalues with respect to which the
map is considered one-resonant depends on the problem one is
facing, in our main result Theorem~\ref{main} it is natural to
consider one-resonance with respect to the set of all
eigenvalues of modulo one.

\section{Normal form for non-degenerate one-resonant diffeomorphisms}\label{3}

Let $F\in \Dif$ be one-resonant with respect to
$\{\lambda_1,\ldots, \lambda_m\}$ with index of resonance $\a$.
Using Poincar\'e-Dulac theory (see, {\sl e.g.} \cite[Chapter
IV]{Ar}), one can formally conjugate $F$ to a germ
$G=(G_1,\ldots, G_n)$ such that
\begin{equation}\label{lam}
G_j(z)=\lambda_j z_j +a_j z^{\alpha k}z_j + R_j(z),\quad j=1,\ldots, m,
\end{equation}
where either $a=(a_1,\ldots,a_m)\neq 0$ and $R_j(z)$ contains
only resonant monomials $a_{js}z^{\a s}z_{j}$ with $s>k$ or
$a_j=0$ and $R_j\equiv 0$ for all $j=1,\ldots, m$. Note that
the second case occurs precisely when $F$ is {\sl formally
linearizable} in the first $m$ variables.

\begin{definition}\label{nondeg}
Let $F\in \Dif$ be one-resonant with respect to
$\{\lambda_1,\ldots, \lambda_m\}$ such that
\begin{equation}\label{lam'}
F_j(z)=\lambda_j z_j +a_j z^{\alpha k}z_j + O(|z|^{|\a|k+2}),\quad j=1,\ldots, m,
\end{equation}
with $k\ge1$ and $a=(a_1,\ldots,a_m)\ne0$, where $\a$ is the index of resonance.
Set
\begin{equation}\label{L}
\Lambda=\L(F):=\sum_{j=1}^m \frac{{a_j} \alpha_j}{{\lambda_j}}.
\end{equation}
We say that $F$ is {\sl non-degenerate} if $\Lambda\neq 0$.
\end{definition}

\begin{remark}\label{com-cam}
The integer $k$  in \eqref{lam'} is invariant under
conjugations preserving the form \eqref{lam'} and the vector
$a=(a_1,\ldots,a_{m})$ is invariant up to multiplication by a
scalar. In particular, the non-degeneracy condition given by
Definition~\ref{nondeg} is invariant. Indeed, if the
conjugation with a map $\psi=(\psi_1,\ldots,\psi_n)\in\Dif$
preserves the form \eqref{lam'} (possibly changing $a$), then
$\psi_j(z)=b_jz_j+O(|z|^2)$, $b_j\in\C^*$, for any
$j=1,\ldots,m$, in view of Remark~\ref{diff}. Conjugating with
the linear part of $\psi$, we see that for any such $j$,  $a_j$
is replaced by $a_j b^{\alpha k}$. Assume now that
$\psi(z)=z+O(|z|^2)$. Then by the Poincar\'e-Dulac theory,
since $\psi$ preseves \eqref{lam'}, all terms of order less
than $|\a|k+2$ that $\psi$ has in its first $m$ components must
be resonant and therefore $a$ is invariant.
\end{remark}

\begin{definition}\label{orderF}
We call the invariant number $k$ the {\sl order} of $F$ with respect to
$\lambda_1,\ldots, \lambda_m$.
\end{definition}

\begin{example}\label{3.4}
Let $F$ be the germ given in Example \ref{one-qp}. Then $F$ is
{\sl non-degenerate} when regarded as a one-resonant germ with
respect to the eigenvalue $1$ (with $k=2$ and $a=a_1=1$). But it becomes {\sl
degenerate} when regarded as a one-resonant germ (with respect to both eigenvalues $\{1, e^{2\pi
i \theta}\}$), because in that case $a=(a_1,a_2)=(0,1)$ and the
index of resonance is $(1,0)$, thus $\Lambda(F)=0$.
The main reason being the change of the order $k$.

Note that, more generally, for a germ of the form $(z+\ldots,
e^{2 \pi i \theta}w+\ldots)$ with $\theta$ irrational, the
condition of being {\sl non-degenerate} with respect to
$\{1,e^{2\pi i \theta}\}$ is equivalent to $F$ being {\sl
dynamically separating} in the terminology of \cite{B-M} (see
Subsection \ref{nonsep}).
\end{example}

As illustrated by the latter example, if one passes from
a smaller set of one-resonant eigenvalues to a larger one,
the order $k$ may drop, in which case the corresponding non-degeneracy
conditions are not related, i.e.\ $F$ can be non-degenerate
with respect to the smaller set but not the larger one
or with respect to the larger but not the smaller one.
On the other hand, if the order $k$ is the same for both sets,
since both sets contain the set of all $\l_j$ with $\a_j\ne0$,
the (non-)degeneracies with respect to the smaller and larger sets
are clearly equivalent.

\begin{remark}\label{non-form-lin}
If $F$ is one-resonant  with respect to
$\{\lambda_1\}$, then $\lambda_1$ is a root of unity.
Moreover, in this case $F$ is non-degenerate if and only if it is not formally linearizable
in the first component.
\end{remark}


We have the following {\em normal form} for non-degenerate partially one-resonant diffeomorphisms.

\begin{theorem}\label{normal}
Let $\in \Dif$ be one-resonant and non-degenerate with respect
to $\lambda_1,\ldots,\lambda_m$ with index of resonace $\a$. Then there exist  $k\in \N$
and numbers $\mu, a_1,\ldots, a_m\in \C$ such that $F$ is
formally conjugated to the map
$\hat{F}(z)=(\hat{F}_1(z),\ldots, \hat{F}_n(z))$, where
\begin{equation}\label{normal-f}
\hat{F}_j(z)=\lambda_j z_j +a_j z^{k\a}z_j +\mu {\alpha_j}{\1\l_{j}}^{-1}z^{2k\a} z_j,
\quad j=1,\ldots, m,
\end{equation}
and the components $\hat{F}_j(z)$ for $j=m+1,\ldots,n$, contain only resonant monomials.
\end{theorem}

\begin{proof}
By the Poincar\'e-Dulac theory, we may assume that $F_j(z)$ for
$j=m+1,\ldots,n$, contain only resonant monomials and
\begin{equation}\label{lam1}
F_j(z)=\lambda_j z_j +a_j z^{k\a}z_j + \sum_{l\ge1} a_{jl} z^{(l+k)\a}z_j ,\quad j=1,\ldots, m.
\end{equation}
With the notation $F':=(F_{1},\ldots,F_{m})$,
$\l':=\hbox{diag}(\l_{1},\ldots,\l_{m})$,
$z':=(z_{1},\ldots,z_{m})$, we can rewrite \eqref{lam1} in the
more compact form
\begin{equation}
F'(z)=\l' z' + z^{k\a}\sum_{j} a_{j}z_{j}e_{j} + \sum_{k'>k} z^{k'\a} \sum_{j} a_{k'j}z_{j}e_{j},
\end{equation}
where the summation over $j$ is understood from $1$ to $m$
and $e_{j}$ is the unit vector with $1$ at the $j$th place and $0$ otherwise.

We now study the conjugation
$
\2F=\Theta\circ F\circ \Theta^{-1}
$
under that map
\begin{equation}\label{theta}
\Theta(z)=z+ \theta(z), \quad \theta(z)=(z^{l\a} \sum_{j} b_{j}z_{j}e_{j},0) =(b_1z^{l\a}z_1,\ldots,b_mz^{l\a}z_m,0,\ldots, 0),
\end{equation}
for an integer $l\ge1$ and a vector $b=(b_{1},\ldots,b_{m})\in\C^{m}$.
We also use the notation
$$F(z)=\l z + f(z), \quad \2F(z)=\l z + \2f(z), \quad f,\2f=O(2),$$
and the Taylor expansions
\begin{equation}\label{exp1}
\2f(z+h)= \2f(z)+ \sum_{r\ge 1} \frac1{r!} \2f^{(r)}(z)(h),
\quad \theta(z+h)=  \theta(z)+\sum_{r\ge 1}\frac1{r!} \theta^{(r)}(z)(h),
\end{equation}
where the derivatives $\2f^{(r)}(z)(h)$ and
$\theta^{(r)}(z)(h)$ are regarded as $n$-tuples of homogeneous
polynomials of degree $r$ in $h$. We use \eqref{exp1}  to rewrite the
identity
\begin{equation}
\2F(\Theta(z)) = \Theta(F(z))
\end{equation}
 as
\begin{equation}\label{main-id0}
\2f(z)  + \l\theta(z)+ \sum_{r\ge1}\frac1{r!}\2f^{(r)}(z) (\theta(z))
= \theta(\l z) + f(z)+
\sum_{r\ge1} \frac1{r!}\theta^{(r)}(\l z) (f(z)).
\end{equation}
In view of the resonance relations, we have $\l\theta(z)=\theta(\l z)$
and hence \eqref{main-id0} is equivalent to
\begin{equation}\label{main-id}
\2f(z) - f(z)=\sum_{r\ge1}\frac1{r!}\left( \theta^{(r)}(\l z) (f(z)) - \2f^{(r)}(z) (\theta(z))\right).
\end{equation}

Now identifying terms of order up to $k|\a|+1$ in \eqref{main-id}, we conclude
by induction on the order that
\begin{equation}\label{f'}
\2f'(z) = f'(z) + O(|z|^{k|\a|+2}) = z^{k\a}\sum_{j} a_{j}z_{j}e_{j} + O(|z|^{k|\a|+2}),
\end{equation}
where $\2f'=(\2f_{1},\ldots,\2f_{m})$.
Next, identifying terms of order up to $(k+l)|\a|+1$, we obtain
$$\2f'(z)-f'(z)=  \theta'^{(1)}(\l z) (f(z)) - \2f'^{(1)}(z) (\theta(z)) + O(|z|^{(k+l)|\a|+2}).$$
Substituting $f'$, $\2f'$ from \eqref{f'} and $\theta$ from \eqref{theta}, we find
\begin{equation}
\2f'(z)-f'(z)=
z^{(k+l)\a}\sum_{j,s} a_{j}b_{s}  \big((l\a_{j}\l^{l\a-e_{j}+e_{s}} + \delta_{js}\l^{l\a})z_{s}e_{s}
- (k\a_{s} + \delta_{js})z_{j}e_{j}
\big)
+ O(|z|^{(k+l)|\a|+2}).
\end{equation}
By the resonance conditions, $\l^{l\a}=1$. In particular, the terms with $\delta_{js}$
cancel each other and we obtain
\begin{multline}\label{final}
\2f'(z)-f'(z)=
z^{(k+l)\a}\sum_{j,s} a_{j}b_{s}  \big(l\a_{j}\l^{e_{s}-e_{j}}z_{s}e_{s}
- k\a_{s}z_{j}e_{j}
\big)
+ O(|z|^{(k+l)|\a|+2})\\
= z^{(k+l)\a} bAZ + O(|z|^{(k+l)|\a|+2}),
\end{multline}
where $b=(b_{1},\ldots,b_{m})$, $Z$ is the diagonal matrix
with entries $z_{1},\ldots,z_{m}$ and $A$ is the $m\times m$
matrix given by
$$A= l ({a}{L}^{-1} \a^{t}) L - k\a^{t}a. $$
Here $\a^{t}$ is the transpose of $\a$ and
$L$ is the diagonal matrix with entries $\l_{1},\ldots,\l_{m}$.
Note that the expression in parentheses is a scalar.
Then
\[
A=CL, \quad  C=  l({a}{L}^{-1} \a^{t}) \id - k\a^{t}{a}{L}^{-1}.
\]

Since the matrix $k\a^{t}{a}{L}^{-1}$ is of rank one, it has
at most one nonzero eigenvalue equal to its trace
$k{a}{L}^{-1}\a^{t}$. By our nondegeneracy assumption, this
trace is actually different from zero. The first matrix in the
expression of $C$ is scalar with all its diagonal entries equal to
$l{a}{L}^{-1}\a^{t}$. Since ${a}{L}^{-1}\a^{t}\ne0$, we
conclude that $C$ is invertible if and only if $l\ne k$. Since
$f$ has the form \eqref{lam1}, given any $l\ne k$, it follows
from \eqref{final} that there exists (unique) vector $b$ such
that $\2f'=f'+O(|z|^{(k+l)|\a|+1})$ and the terms of $\2f'$ of
order $(k+l)|\a|+1$ all vanish.

On the other hand, in case $l=k$, $C$ has rank $m-1$.
In this case we use the identity
$$
bC\a^{t} = kb\left({a}{L}^{-1}\a^{t}\right)\a^{t} - kb\left(\a^{t}{a}{L}^{-1}\right)\a^{t}
= k({a}L^{-1} \a^{t})b\a^{t} - k b\a^{t} ({a}{L}^{-1}\a^{t})=0,
$$
where we have used that both $b$ and $b\a^{t}$ commute with the scalar ${a}L^{-1}\a^{t}$.
Hence $\a^{t}$ annihilates the image of the map $b\mapsto bC$.
Since $C$ has rank $m-1$, its image is precisely the orthogonal complement of $\a$
(with respect to the standard hermitian scalar product on $\C^{m}$, note that $\1\a=\a$).
Then the image of the map $b\mapsto bA$ is precisely the orthogonal complement
of ${\a}{\1L}^{-1}=({\a_{1}}{\1\l_{1}}^{-1},\ldots,{\a_{m}}{\1\l_{m}}^{-1})$.
It now follows from \eqref{final} that, choosing suitable $b$,
we can arrange that the term of $\2f'$ of order $2k|\a|+1$
equals
\begin{equation}\label{2k}
z^{2k\a}\mu{\a}{\1L}^{-1}Z = z^{2k\a}\mu \sum_{j}{\a_{j}}{\1\l_{j}}^{-1}z_{j}e_{j}
\end{equation}
for some $\mu\in\C$.

We now apply inductively the above procedure for each $l\ge1$, either to eliminate
the corresponding term in \eqref{lam1} or normalize it as in \eqref{2k},
by conjugating with a suitable map \eqref{theta} for that number $l$.
At each step we may create nonresonant terms whose order
must be greater than $l|\a|+1$ in view of \eqref{main-id0}.
Those terms can be eliminated inductively according to the  Poincar\'e-Dulac theory
by conjugation with further maps $\Theta(z)=z+\theta(z)$ with $\theta(z)$ being suitable monomials
of order greater than $l|\a|+1$. Again using \eqref{main-id0}
we see that those additional conjugations does not affect the normalized
terms of order $l|\a|+1$. Thus by induction on $l$, we obtain the desired normalization
\eqref{normal-f}.
\end{proof}

\begin{remark}
It is clear from the proof of Theorem \ref{normal} that, for
any given $t\in \N$ there exists a holomorphic (polynomial)
change of coordinates which transforms $F$ into $\hat{F}+O(t)$,
where $\hat{F}$ satisfies \eqref{normal-f} and $O(t)$ denotes a
function vanishing of order $\geq t$ at $0$.
\end{remark}

\section{Dynamics of normal forms}\label{five}

Motivated by Theorem~\ref{normal}, we shall first study the dynamics of a one-resonant
diffeomorphism $G\in\Dif$  of the form $G(z)=(G_1(z),\ldots,
G_n(z))$ with
\begin{equation}\label{f-normj}
G_j(z)=\lambda_jz_j+a_j z^{k\a} z_j +b_j z^{2k\a}z_j,
\quad j=1,\ldots,n,
\end{equation}
and $\L=\Lambda(G)\neq 0$, where $\L(G)$ is as in Definition~\ref{nondeg}.
We consider the singular foliation $\mathcal F$ of $\C^{n}$ given by
$\{z^\a=\hbox{const}\}$.

\begin{lemma}
The foliation $\mathcal F$ is $G$-invariant.
\end{lemma}

\begin{proof}
Let $\pi(z)=z^\a$. Then
\begin{equation*}
\pi(G(z))=z^\a \prod_{j=1}^n (\l_j + a_j z^{k\a} +b_j z^{2k\a})^{\a_{j}},
\end{equation*}
where the right-hand side is clearly a holomorphic function of $z^\a$.
 Hence $G$ maps leaves of $\mathcal F$
into (possibly different) leaves of $\mathcal F$
and the desired conclusion follows.
\end{proof}

Let $\mathcal L$ denote  the space of leaves of $\mathcal F$.
Let $\pi:\C^n\to \mathcal L$ be the projection given by
$(z_1,\ldots, z_n)\mapsto z^\a$. Clearly, $\mathcal L\simeq
\C$. Let $u=z^\a=\pi(z)$. The action of $G$ on $\mathcal L$ is
given by
\begin{equation}\label{phi-u}
\Phi(u):=G_1(z)^{\a_1}\cdots G_n(z)^{\a_n}=u+\Lambda(G) u^{k+1} +O(|u|^{k+2}),
\end{equation}
where we have used that $\l^\a=1$. Note that $\Phi:(\C,0)\to (\C,0)$ is
locally biholomorphic.  Let $v_1,\ldots, v_k$ be the attracting directions for $\Phi$,
and $P_j\subset\C$, $j=1,\ldots,k,$ attracting
petals  centered at $v_j$ (see Section~\ref{preli}). Set
\[
U_j:=\pi^{-1}(P_j)\subset\C^{n}.
\]
Since $\mathcal F$ is $G$-invariant, the domains $U_j$ are also
$G$-invariant.

Let $z\in U_j$. Then $\Phi^{\circ m}(\pi(z))\to 0$ as $m\to
\infty$. In order to understand the dynamics of $G$, it is then
sufficient to understand the ``motion'' along the leaves of
$\mathcal F$.  As a matter of notation, let $p_j(z_1,\ldots, z_n)=z_j$.

\begin{proposition}\label{norm-dynamics}
Let $G\in\Dif$ be in the normal form \eqref{f-normj} with $\L=\Lambda(G)\neq
0$. Fix $1\le j\le n$ and $1\le t\le k$.
\begin{enumerate}
\item If $|\lambda_j|<1$, then for all $z\in U_t$, one has $\lim_{m\to\infty} p_j\circ
  G^{\circ m}(z)=0$.
  \item If $|\lambda_j|>1$, then for all $z\in U_t$ with $z_{j}\ne0$, one has $\lim_{m\to\infty} p_j\circ
  G^{\circ m}(z)=\infty$.
  \item If $|\lambda_j|=1$ and $\Re ({a_j}{\lambda_j}^{-1}\L^{-1}) >0$, then
  for all $z\in U_t$, one has $\lim_{m\to\infty} p_j\circ
  G^{\circ m}(z)=0$.
  \item If $|\lambda_j|=1$ and  $\Re ({a_j}{\lambda_j}^{-1}\L^{-1}   ) <0$, then
  for all $z\in U_t$ with $z_{j}\ne0$, one has
  $\lim_{m\to\infty} p_j\circ G^{\circ m}(z)=\infty$.
\end{enumerate}
\end{proposition}

\begin{proof}
Note that by construction $\Phi(\pi(z))=\pi(G(z))$. We can
write for $j=1,\ldots, n$,
\[
G_j(z)=(\lambda_j+a_j u^k +b_j u^{2k})z_j
\]
and, letting $u_l:=\Phi^{\circ l}(u)=\pi(G^{\circ l}(z))$,
\[
p_j \circ G^{\circ m}(z)=\lambda_j^m\prod_{l=1}^m
\left(1+\frac{a_j}{\lambda_j} u_l^k +\frac{b_j}{\lambda_j} u_l^{2k}\right)z_j.
\]
We examine the asymptotical behavior of the infinite product
\begin{equation}\label{inf-p}
\prod_{l=1}^\infty
\left(1+\frac{a_j}{\lambda_j} u_l^k +
\frac{b_j}{\lambda_j} u_l^{2k}\right).
\end{equation}
Let $z\in U_t$, therefore $u=\pi(z)\in P_t$.
By Theorem \ref{LF}, part (2), it follows that $|u^k_l|=|u_l|^k\sim
\frac{1}{l}$.

Let $A_l:=\frac{a_j}{\lambda_j}u_l^k +
\frac{b_j}{\l_j}u_l^{2k}$. We examine the behavior of
$\prod_{l=1}^m|1+A_l|$. Taking the logarithm  we have
\[
\log\left( \prod_{l=1}^m |1+A_l|\right)=\frac{1}{2}\sum_{l=1}^m \log (1+|A_l|^2+2\Re A_l).
\]
For $l>>1$, $||A_l|^2+2\Re A_l|\sim l^{-c}$ for some $c\in
\N^\ast$. Hence, since for $l>>1$, $\log (1+|A_l|^2+2\Re
A_l)\sim |A_l|^2+2\Re A_l$,
\[
\frac{1}{2}\sum_{l=1}^\infty |\log (1+|A_l|^2+2\Re A_l)|\sim  \frac{1}{2}\sum_{l=1}^\infty l^{-c}.
\]
From this it follows that the infinite product \eqref{inf-p}
either converges or goes to zero or infinity much slower than
$|\lambda_j|^m$ in case $|\lambda_j|\neq 1$. Thus (1) and (2)
follow.

As for (3) and (4) we need a better estimate. By Theorem
\ref{LF}, part (1), it follows that $\frac{u_l^k}{|u_l|^k}\to
v_t^k=-|\L|\L^{-1}$ as $l\to \infty$. Hence
\[
\lim_{l\to\infty}\Re \left(\frac{a_j}{\lambda_j} \frac{u_l^k}{|u_l|^k} \right)=\Re \left(\frac{a_j}{\lambda_j} v_t^k \right)=-\Re \left(\frac{a_j}{\lambda_j} \frac{|\L|}{\L} \right).
\]
Therefore in case (3), for $l$ large, $|A_l|^2+2\Re A_l\sim
(-l^{-1})$. Hence
\[
\log\left( \prod_{l=1}^\infty |1+A_l|\right)=\frac{1}{2}\sum_{l=1}^\infty \log (1+|A_l|^2+2\Re A_l)\sim \frac{1}{2}
\sum_{l=1}^m \frac{-1}{l}=-\infty,
\]
and thus (3) follows. Statement (4) is similar.
\end{proof}

\section{Dynamics of non-degenerate one-resonant
maps}\label{prova}

\begin{definition}
Let $F\in\Dif$ be  one-resonant  and non-degenerate with
respect to $\{\lambda_1,\ldots,\lambda_m\}$.
Let $k\in \N$ be the order of $F$ with respect to
$\lambda_1,\ldots, \lambda_m$ (see Definition \ref{orderF}).
Choose coordinates such that \eqref{lam'} holds. We say that
$F$ is {\sl parabolically attracting} with respect to
$\{\lambda_1,\ldots, \lambda_m\}$ if
\begin{equation}\label{parab}
|\l_{j}|=1, \quad
\Re\left( {a_j}{\lambda_j}^{-1}\L^{-1}\right)>0, \quad j=1,\ldots, m,
\end{equation}
where $\L=\L(F)$ is given by \eqref{L}
\end{definition}

\begin{remark}
The condition of being parabolically attracting is independent
of the coordinates chosen. To see this, let $\psi$ be a
transformation which preserves \eqref{lam'}, and let
$\tilde{F}:=\psi\circ F\circ \psi^{-1}$. In view of Remark
\ref{com-cam}, it suffices to check the invariance of
\eqref{parab} for $\psi$ linear with $\psi_{j}(z)=b_{j}z_{j}$,
$b_{j}\in\C^{*}$, for any $j=1,\ldots, m$. Then, $a_j$ is
replaced by $\tilde{a}_j:=a_j b^{\a k}$ and
$\Lambda(\tilde{F})=\Lambda(F)b^{\a k}$ from which the claim follows.
\end{remark}

\begin{remark}\label{radice}
If $F$ is one-resonant and non-degenerate with respect to
$\{\lambda_1\}$ (with $|\lambda_1|=1$), then it is always
parabolically attracting. Indeed, in such a case,
$\Lambda={a_{1}\a_{1}}{\l_{1}}^{-1}$ and
\[
\Re\left( {a_1}{\lambda_1}^{-1}\L^{-1}\right)=\a_{1}^{-1}>0.
\]
\end{remark}

\begin{definition}
Let $F\in\Dif$. We call a {\sl basin of attraction for $F$ at
$0$} a nonempty (not necessarily connected) open set
$U\subset\C^n$ with $0\in \1U$, for which there exists a
neighborhood basis $\{\Omega_j\}$ of $0$ such that $F(U\cap
\Omega_j)\subset U\cap \Omega_j$ and $F^{\circ m}(z)\to 0$ as
$m\to \infty$  whenever $z\in U\cap \Omega_j$ holds for some $j$.
\end{definition}

We are now ready to give the proof of Theorem~\ref{main}.

\begin{proof}[Proof of Theorem~\ref{main}]
Denote $u:=z^\a$. In view of Theorem~\ref{normal}, up to biholomorphic conjugation we can
assume that $F(z)=(F_1(z),\ldots, F_n(z))$ with
\begin{equation}\label{formF}
\begin{split}
F_j(z)&=(\lambda_j+a_j u^k+\mu \lambda_j\a_j u^{2k})z_j + O(|z|^l),\quad j=1,\ldots, m,\\
F_j(z)&=\lambda_j z_j + O(|z|^2), \quad j=m+1,\ldots, n,
\end{split}
\end{equation}
for any fixed $l$ to be chosen later. Also, acting with a dilation (cfr. Remark
\ref{com-cam}) we can assume that $\L=\Lambda(F)=-1/k$. Then, since $F$ is parabolically
attracting, we have
\begin{equation}\label{Dell}
\Re\left( {a_j}{\lambda_j}^{-1}\right)<0, \quad j=1,\ldots, m.
\end{equation}

Let  $R>0$ be a number we will suitably choose later. Let
\[
\Delta_R:=\left\{u\in \C: \left|u^k-\frac{1}{2R}\right|<\frac{1}{2R}\right\}.
\]
Note that $\Delta_R$ has exactly $k$ connected components
corresponding to different branches of the $k$th root. The
desired basins of attraction will be constructed by means of
the projection $z\mapsto z^\a$ over sectors contained in such
connected components.

We first construct a basin of attraction based on a sector
centered at the direction $1$, namely,

\begin{equation}\label{choose-sector}
S_R(\eps):=\{u\in \Delta_R: |{\sf Arg} u|<\epsilon\},
\end{equation}
for some small $\epsilon>0$ to be chosen later.

Let $\beta>0$ be such that $\beta |\a|<1$ and let
\[
B:=\{z=(z_1,\ldots, z_n)\in \C^n : |z_j|<|u|^\beta, j=1,\ldots, m, |(z_{m+1},\ldots, z_n)|<|u|^\beta, u:=z^\a\in S_R( \epsilon)\}.
\]
First of all, $B\neq \emptyset$ and $0\in \partial B$. Indeed,
it is easy to see that   $z_r=(r,\ldots, r)\in B$ for $r>0$
sufficiently small. Moreover, since the map $z\mapsto z^\a$ is
open and $0$ is not in the interior of $S_R(\eps)$, it follows
that $0\notin B$, i.e.\ $0\in\partial B$. Finally, the set $B$
is obviously open.

Next, we prove that $B$ is $F$-invariant. Let $z\in B$ and let
$u:=z^\a$. Let
\[
\Phi(u,z):=F_1^{\a_1}(z)\cdots F_m^{\a_m}(z)=u-\frac{1}{k}u^{k+1}+h_1(u)+h_2(z),
\]
where we consider $\Phi$ as a function of the variables $z,
u=z^\a$ and $h_1(u)=O(|u|^{k+2})$ and $h_2(z)=O(|z|^l)$.
We make the change of coordinates $U=u^{-k}$
and write
$$\2\Phi(U,z):=\Phi(U^{-\frac1{k}},z)^{-k}$$
for the map $\Phi$ in the new coordinates. Note that $u\in
S_R(\epsilon)$ if and only if $U\in H_R(\epsilon)$, where
\[
H_R(\epsilon):=\{w\in \C: \Re w>R, |{\sf Arg} w|<k\epsilon\}.
\]
Since $\Re( a_j\l_j^{-1})<0$, it is easy to see that, choosing $R$
sufficiently large and $\eps$ sufficiently small, we obtain
\begin{equation}\label{estimo}
\left|1+\frac{a_j}{\lambda_j} \frac{1}{w} +\mu \a_j\frac{1}{w^2}   \right|< 1-\frac{c}{|w|}
\end{equation}
for some $c>0$ and for all $w \in H_R(\eps)$.

Now fix $0<\delta<1/2$ such that
\begin{equation}\label{incl}
H_{R}(\epsilon)+1+\tau\subset H_{R}(\epsilon) \quad \text{ whenever } \quad |\tau|<\delta,
\end{equation}

Note that $\delta$ depends on $\epsilon$ but not on $R$. Fix
$0<c'<c$. By choosing $\beta<1/2$ sufficiently small, we can
assume that
\begin{equation}\label{servizio}
\beta(\delta+1)-c'k<0
\end{equation}
and choose $l>1$ such that
\begin{equation}\label{bk}
\beta l>k+1.
\end{equation}

After a direct computation we find
\begin{equation}\label{phio}
\2\Phi(U,z)=U\left( \frac{1}{1-\frac{1}{kU}+ U^{1/k} h_1(U^{-1/k})
+U^{1/k}{h_2(z)}}\right)^k.
\end{equation}
Since $|z|<n|u|^\beta$ in $B$, there exists $K>0$ such that
\[
|U|^{1/k}{|h_1(U^{-1/k})|} \leq K |U|^{1/k} |U|^{-(k+2)/k} = K |U|^{-1-1/k}
\]
and
\[
|U|^{1/k}{|h_2(z)|}\leq K |U|^{1/k} |u|^{\beta l}=K |U|^{(1-\beta l)/k}.
\]
 Therefore,  if $R$ is
sufficiently large and $z\in B$ (hence $U\in H_R(\epsilon)$),
we have
\begin{equation}\label{nu}
\2\Phi(U,z) = U+1+\nu(U,z), \quad \hbox{with\ }|\nu(U,z)|<\delta,
\end{equation}
where we have used \eqref{bk}.
In particular, $U_1:=\tilde{\Phi}(U,z)\in H_R(\epsilon)$ in view of \eqref{incl} and
$\Re U_1\geq \Re U+\frac{1}{2}$. Therefore we have proved that
\begin{equation}
z\in B \Rightarrow u_1:=\Phi(u,z)\in S_R(\epsilon).
\end{equation}
Moreover, by the same token, setting by induction
$u_{m+1}:=\Phi(u_{m},F^{\circ m}(z))$, it follows that
\begin{equation}\label{goinf}
\lim_{m\to \infty} u_m=0.
\end{equation}

Now we examine the components  $F_j$ for $j=m+1,\ldots, n$.
Set $x:=(z_1,\ldots, z_m)$ and $y:=(z_{m+1},\ldots, z_n)$.
Then
\[
y_1=My+h(z) z,
\]
where $M$ is the $(n-m)\times (n-m)$ diagonal matrix with
entries $\lambda_j$ ($j=m+1,\ldots, n$) and $h$ is a
holomorphic $(n-m)\times n$ matrix valued function in a
neighborhood of $0$ such that $h(0)=0$. If $z\in B$, then
$|y|<|u|^\beta$. Moreover, since $|\lambda_j|<1$ for
$j=m+1,\ldots, n$, it follows that there exists $a<1$ such that
$|My|<a|y|<a|u|^\beta$. Also, let $0<b<1-a$. Then, for $R$
sufficiently large, it follows that $|h(z)|\leq b/n$ if $z\in
B$. Hence, letting $p=a+b<1$, we obtain
\begin{equation}
\label{p1}
|y_1|\leq |My|+|h(z)| |z|<a|u|^\beta+\frac{b}{n} n|u|^\beta= (a+b)|u|^\beta=p |u|^\beta.
\end{equation}
Now, we claim that for $R$ sufficiently large, it follows that
\begin{equation}\label{plimio}
|u|\leq \frac{1}{p^{1/\beta}}|u_1|,
\end{equation}
where $u_1=\Phi(u,z)$. Indeed, \eqref{plimio} is equivalent to
$|U_1|\leq p^{-k/\beta} |U|$ and hence to
\[
\frac{|U+1+\nu(U,z)|}{|U|}\leq p^{-k/\beta}.
\]
But the limit for $|U|\to \infty$ in the left-hand side is $1$
and the right-hand side is $>1$, thus \eqref{plimio} holds
for $R$ sufficiently large.

Hence, by \eqref{p1} and \eqref{plimio} we obtain
\begin{equation}
\label{p2}
|y_1|\leq  |u_1|^\beta.
\end{equation}

Now we examine the components $F_j$ for $j=1,\ldots, m$. Let
$z\in B$ and, as before, let $U=u^{-k}\in H_R(\epsilon)$. By
\eqref{formF} and by \eqref{Dell} we have
\begin{equation}\label{fj}
F_j(z)=\lambda_j\left(1+\frac{a_j}{\lambda_j} \frac{1}{U}+\mu \a_j\frac{1}{U^2}\right)z_j+R_l(z)
\end{equation}
with $R_l(z)=O(|z|^l)$. If $z\in B$ and $R$ is sufficiently
large, one has
\begin{equation}\label{unog}
|R_l(z)|< |z|^{l-1} < (n|u|^{\beta})^{ (l-1)}.
\end{equation}
From \eqref{fj}, and since $U\in H_R(\eps)$, we now obtain
using \eqref{estimo} and \eqref{unog}:
\[
|F_j(z)|\leq \left(1-\frac{c}{|U|}+n^{l-1}|u|^{\beta (l-2)}\right)|u|^\beta=
\left(1-\frac{c}{|U|}+\frac{n^{l-1}}{|U|^{\beta (l-2)/k}}\right)|u|^\beta.
\]
Then
$\beta(l-2)>k+1-2\beta>k$ by \eqref{bk} and thus $\beta(l-2)/k>1$. Hence, if
$R$ is sufficiently large,
\begin{equation}\label{p'}
p(u):=1-\frac{c}{|U|}+\frac{n^{l-1}}{|U|^{\beta
(l-2)/k}}<1.
\end{equation}
Hence
\begin{equation}\label{f2}
|F_j(z)|\leq p(u) |u|^\beta.
\end{equation}

Now we claim that, setting $u_1:=\Phi(u,z)$, we obtain
\begin{equation}\label{Fp}
|u|\leq \frac{1}{|p(u)|^{1/\beta}}|u_1|.
\end{equation}
Indeed, \eqref{Fp} is equivalent to $|U_{1}|\le
p(u)^{-k/\b}|U|$ and hence, in view of \eqref{nu}, to
\begin{equation}\label{est}
\frac{|U+1+\nu(U,z)|}{|U|}\leq p(u)^{-k/\beta}.
\end{equation}
Note that, since $0<c'<c$ (recall that $c'$ is chosen before
\eqref{servizio}), taking $R$ sufficiently large,
$(1-c'/|U|)^{-k/\beta}\leq p(u)^{-k/\beta}$. Also, by
\eqref{nu} we have $\frac{|U+1+\nu(U,z)|}{|U|}\leq
1+\frac{1+\delta}{|U|}$, hence \eqref{est} holds if we can show
that
\begin{equation}\label{new-deal}
1+\frac{1+\delta}{|U|}\leq \left(1-\frac{c'}{|U|}\right)^{-k/\beta}.
\end{equation}
But
\[
\left(1-\frac{c'}{|U|}\right)^{-k/\beta}=1+\frac{k}{\beta}\frac{c'}{|U|}+o\left(\frac{1}{|U|} \right),
\]
and thus \eqref{new-deal} holds whenever $\delta+1-c'k/\beta<0$
(which is assured by \eqref{servizio}) and $R$ is sufficiently
large. Hence, \eqref{Fp} holds. Putting together \eqref{f2} and
\eqref{Fp} we have for $j=1,\ldots, n$
\begin{equation}\label{Fins}
|F_j(z)|\leq |u_1|^\beta, \quad j=1,\ldots,m.
\end{equation}
Equations \eqref{p2} and \eqref{Fins} imply that
$F(B)\subseteq B$. Moreover, by induction, for all $z\in B$,
denoting by $\rho_j(z):=z_j$ the projection on the $j$-th
component, we have
\[
|\rho_j \circ F^{\circ m}(z)|\leq |u_m|^\beta,
\]
hence $F^{\circ m}(z)\to 0$ as $m\to \infty$ by \eqref{goinf}.
Therefore $B$ is a basin of attraction for $F$ at $0$.

To end the proof, we note that the previous argument can be
repeated by considering in \eqref{choose-sector}  the sectors
$S^j_R( \epsilon)$, $j=1,\ldots, k$,  of the form
\[
S^j_R(\epsilon):=\{u\in \Delta_R: |{\sf Arg} u - \frac{2\pi(j-1)}{k} | <\epsilon\}.
\]
Let $B_j$ be the basin of attraction constructed over
$S^j_R(\epsilon)$, namely
\[
B_{j}:=\{z=(z_1,\ldots, z_n)\in \C^n : |z_j|<|u|^\beta, j=1,\ldots, m, |(z_{m+1},\ldots, z_n)|<|u|^\beta, u:=z^\a\in S^{j}_R( \epsilon)\}.
\]
Then clearly  $B_1, \ldots, B_k$
 are disjoint and the proof is complete.
\end{proof}

\begin{remark}\label{proietto}
Let $F$ be as in Theorem \ref{main} and let $B_1,\ldots,
B_k$ be its basins of attraction at $0$ constructed in the proof. If $S_1, \ldots, S_k$
denote the $k$ petals for the induced germ $u\to u+\Lambda
u^{k+1}+O(|u|^{k+2})$ (see \eqref{phi-u}) then, up to
relabeling, $\pi(B_j)\subset S_j$ for $j=1,\ldots, k$, where
$\pi:\C^n \ni z \mapsto z^\a \in \C$. In particular, if
$\a=(q,0,\ldots, 0)$ for some $q\geq 1$, then each $B_k$ has
$q$ connected components.
\end{remark}

\section{Applications and examples}\label{appl}

\subsection{Semi-attractive case} One-resonant diffeomorphisms
with respect to one eigenvalue (which is necessarily a root of unity) are either formally
linearizable in the associated eigendirection or
non-degenerate and parabolically attracting (see Remark~\ref{radice}).
Thus, in particular, we recover
Hakim's theorem on semi-attractive germs (cfr. \cite[Thm.
1.1]{H} for $q=1$):

\begin{corollary}\label{hakim}
Let $F\in \Dif$. Let $\{\lambda_1,\ldots, \lambda_n\}$ be the
eigenvalues of $dF_0$. Suppose that $\lambda_1^q=1$ for some
$q\in \N\setminus\{0\}$ and $\lambda_1^l\neq 1$ for
$l=1,\ldots, q-1$, and that $|\lambda_j|<1$ for $j=2,\ldots,
m$. In particular, $F$ is one-resonant with respect to
$\{\l_{1}\}$. Let $k$ be the order of $F$ with respect to
$\lambda_1$. Then:
\begin{enumerate}
\item either $k<\infty$ and there exist $k$  basins of attraction for $F$ at $0$,
each having $q$ connected components which are cyclically
permuted by $F$,
  \item or $k=\infty$ and $F$ is formally linearizable
  in the first component. This is the case if and only if
  there exists a holomorphic germ of a non-singular  curve of
  fixed points of $F^{\circ q}$ passing  through $0$.
\end{enumerate}
\end{corollary}

\begin{proof} (1) If $k<\infty$ then $F$ is non-degenerate
with respect to $\{\lambda_1\}$ and with index of resonance
$(q,0,\ldots, 0)$ (see Remark \ref{non-form-lin}). By Remark
\ref{radice}, $F$ is parabolically attracting with respect to
$\{\lambda_1\}$ and hence Theorem \ref{main} applies
yielding $k$ basins of attraction.
Let $B_{1},\ldots,B_{k}$ be the basins of attraction
constructed in course of the proof of Theorem~\ref{main}.
By Remark~\ref{proietto}, each $B_{j}$ has $q$ connected
components.

Fix one such a basin of attraction $B=B_{j}$ and let $D_0,\ldots,
D_{q-1}$ be its connected components. By Remark~\ref{proietto},
 the image of $B$ in $\C$ via the map $\C^n\ni
z\mapsto z_1^q$ belongs to a petal
 $S$ of $u\mapsto u+\Lambda(F)u^{k+1}$. Let  $w\in \C$ be such that $w^q\in S$.
 In view of the construction in the proof of Theorem~\ref{main},
 assuming $w$ being sufficiently small,
 we have
$Q_p:=(\lambda_1^pw, 0,\ldots, 0)\in B$ for $p=0,\ldots, q-1$.
Moreover, the $Q_p$'s belong to different connected components
of $B$. We can assume $Q_p\in D_p$ for $p=0,\ldots, q-1$. Now
$F_1(Q_p)=\lambda_1^{p+1}w+o(|w|)$ and hence $F(Q_p)$ belongs
to $D_{p+1}$ (where $D_{q}=D_0$), proving the statement.

(2) We note that by Definition~\ref{orderF}, the orders of $F$ and of
$F^{\circ q}$ with respect to $\lambda_1$ coincide. Furthermore,
$k=\infty$ if and only if $F$ is formally
linearizable in the first component, and hence, if and only if
$F^{\circ q}$ is formally linearizable in the first component.
Therefore we can assume $q=1$. One direction being clear,
we only show that if $k=\infty$ then $F$ has a holomorphic
non-singular curve of fixed points through $0$.  Write $z=(z,z')\in
\C\times \C^{n-1}$, $\lambda'=(\lambda_2,\ldots,\lambda_n)$ and
\[
F(z,z')=(f(z,z'), \lambda'z'+ g(z,z')) \in \C\times \C^{n-1},
\]
 where $f(z,z')=z+o(|(z,z')|)$ and $g(z,z')=o(|(z,z')|)$. We
look for a curve given by $\psi:\zeta\mapsto (\zeta, v(\zeta))$
where $v:U\to \C^{n-1}$ is a germ at $0$ of holomorphic map defined in
some open set $U\subset \C$ such that $v(0)=0$
and such that $F(\psi(\zeta))=\psi(\zeta)$ for all $\zeta\in
U$. We decouple the latter condition as
\begin{equation}\label{dec1}
f(\zeta, v(\zeta))=\zeta
\end{equation}
\begin{equation}\label{dec2}
\lambda' v(\zeta) + g(\zeta, v(\zeta))=v(\zeta)
\end{equation}
Since $k=+\infty$, the Poincar\'e-Dulac theory yields that $F$
is formally conjugated to a map of the type $\hat{F}(z,z')=(z,
\lambda' z' + h(z,z'))$, where each monomial in the expansion of $h(z,z')$ is
divisible by $z_j$ for some $j=2,\ldots, n$. Clearly $\hat{F}$
has a unique curve of fixed points tangent to $e_1$, namely
$z'=0$. Hence, $F$ has a unique formal solution to \eqref{dec1}
and \eqref{dec2}. It is enough to show that such a solution is
actually holomorphic. To this aim, we let
$G(x,y):=(\lambda'-\id)y+g(x,y)$ with $x\in \C$ and $y\in
\C^{n-1}$. Since the Jacobian matrix $\{\frac{\partial
G_j(x,y)}{\partial y_k}|_{0}\}_{j,k=1,\ldots,
n-1}=\lambda'-\id$ has maximal rank, then by the (holomorphic)
implicit function theorem, there exists a unique function $v(x)$
defined and holomorphic near $x=0$ such that $G(x,v(x))\equiv
0$, and the proof is complete.
\end{proof}

\subsection{Quasi-parabolic germs.}\label{nonsep}
A germ of holomorphic diffeomorphism of $\C^{2}$ at $0$ of the
form $F(z,w)=(z+\ldots, e^{2 \pi i \theta} w+\ldots)$ with $\theta\in\R$ is called
{\sl quasi-parabolic}. In particular, if $\theta\in
\R\setminus{\mathbb Q}$, then $F$ is one-resonant.
We shall restrict to this case here.

Using Poincar\'e-Dulac theory, since all resonances are of the
type $(1,(m,0))$, $(2, (m,1))$, the map $F$ can be formally
conjugated to a map of the form
\begin{equation}\label{form-qp}
\hat{F}(z,w)=(z+\sum_{j=\nu}^\infty a_j z^j, e^{2 \pi i \theta} w+\sum_{j=\mu}^\infty b_j z^j w),
\end{equation}
where we assume that either $a_\nu\neq 0$ or $\nu=\infty$  if
$a_j=0$ for all $j$. Similarly for $b_\mu$.

As it is proved in \cite{B-M}, the number $\nu(F):=\nu$ is a
formal invariant of $F$. Moreover, it is proved that, in case
$\nu<+\infty$, the sign of $\Theta(F):=\nu-\mu-1$ is a formal
invariant. The map $F$ is said {\sl dynamically separating} if
$\nu<+\infty$ and $\Theta(F)\leq 0$.

An argument similar to that of the proof of Corollary
\ref{hakim}.(2) yields:

\begin{proposition}
Let $F$ be a quasi-parabolic germ of diffeomorphism of $\C^2$
at $0$. Then $\nu(F)=+\infty$ if and only if there exists a
  germ of (holomorphic) curve through $0$
  that consists of fixed points of $F$.
\end{proposition}

In case $\nu(F)<+\infty$, we note that the index $\a=(1,0)$
and therefore $\Lambda(F)$ equals either $a_{\nu(F)}$ or $0$,
depending on whether $\nu\le\mu+1$ or $\nu >\mu+1$.
Hence $F$ is dynamically separating if and only if it is
non-degenerate with respect to $\{1,e^{2\pi i \theta}\}$. In case
$F$ is non-degenerate, $k:=\nu(F)-1$ is the order of $F$ with
respect to $\{1, e^{2\pi i\theta}\}$.

In \cite{B-M} it is proven that if $F$ is a quasi-parabolic
dynamically separating germ of diffeomorphism at $0$ then there
exist $\nu(F)-1$ petals for $F$ at $0$.
A direct computation shows that if $F$ is dynamically separating,
then it is parabolically attracting
if and only if
\begin{equation}\label{cond1}
\Re
\left(\frac{b_{\nu-1}}{e^{2\pi i\theta} a_{\nu}}
\right)>0.
\end{equation}
Then as a consequence of Theorem \ref{main} and Remark~\ref{proietto} we have:

\begin{corollary}
Let $F$ be a dynamically separating quasi-parabolic germ,
formally conjugated to \eqref{form-qp}. If \eqref{cond1} holds,
then there exist $\nu(F)-1$ disjoint connected basins of
attraction for $F$ at $0$.
\end{corollary}

\subsection{An example of an elliptic germ with parabolic dynamics}\label{mixed}
Let $\lambda=e^{2\pi i \theta}$ for some $\theta\in
\R\setminus{\mathbb Q}$. Let
$$F(z,w)=(\lambda z + az^2w+\ldots,
\lambda^{-1}w +bzw^2+\ldots),$$
 with $|a|=|b|=1$. Then $F$ is one-resonant
with index of resonance $(1,1)$
and for
each choice  of $(a,b)$ such that the germ  is non-degenerate (i.e.\ $a\l^{-1}+b\l\ne0$),
 there exists a basin of
attraction for $F$ at $0$.
Indeed, it can be checked that the non-degeneracy condition implies
that $F$ is parabolically attracting with respect to $\{\l,\l^{-1}\}$
and hence Theorem~\ref{main} applies.

A similar argument can be applied to
$F^{-1}$, producing a basin of repulsion for $F$ at $0$.
Hence we have a parabolic type dynamics for $F$.

On the other hand, suppose further that there exist $c>0$ and
$N\in \N$ such that $|e^{2\pi q i \theta}-1|\geq c q^{-N}$ for
all $q\in \N\setminus\{0\}$ (such a condition holds for
$\theta$ in a full measure subset of the unit circle). Since
$\lambda^q\neq \lambda$ for all $q\in \N$, it follows from
\cite[Theorem 1]{Po} that there exist two analytic discs
through $0$, tangent at the origin to the $z$-axis and to the
$w$-axis respectively, which are $F$-invariant and such that
the restriction of $F$ on each such a disc is conjugated to
$\zeta\mapsto \lambda \zeta$ or $\zeta\mapsto \lambda^{-1}
\zeta$ respectively. Thus, in such a case, the elliptic and
parabolic dynamics mix, although the spectrum of $dF_0$ is only
of elliptic type.

\subsection{Examples of one-resonant degenerate germs with no basins of attraction}\label{nobase}
Set
\begin{equation}\label{esempio}
F(z,w)=\left(\lambda z\left(1-\frac{zw}\l\right)^{{-1}}, \frac{w}{\lambda}\left( 1 -
\frac{zw}{\lambda}\right)\right),
\end{equation}
with $|\lambda|=1$ and $\lambda$ not a root of unity.
Then $F$ is
one-resonant with index of resonance $\a=(1,1)$ but it is
degenerate because
\[
\Lambda(F)=\frac{1}{\lambda}-\frac{1}{\lambda^2}\cdot \lambda=0.
\]
Note also that the order of $F$ is $k=1$. Set $u=zw$ and
\[
\Phi(u)=F_1(z,w)\cdot F_2(z,w)=u.
\]
We claim that $F$ has no basins of attraction at $0$.
Indeed, suppose $F^{\circ n}(z,w)\to 0$ as $n\to\infty$ for some $(z,w)$.
Then it follows that $\Phi^{\circ n}(zw)\to 0$ as $n\to\infty$,
which implies that $zw=0$. The latter cannot hold on a nonempty open set.

A less trivial example demonstrating this phenomenon is the following.
Set
\begin{equation}\label{esempio1}
F(z,w)=\left(\lambda z + z^2w, \frac{1}{\lambda}w -
\frac{1}{\lambda^2}zw^2\right),
\end{equation}
where $|\lambda|=1$ and $\lambda$ is not a root of unity.
As before, $F$ is one-resonant with index of resonance  $(1,1)$
and $\L(F)=0$. The order of $F$ is $1$ and
for $u=zw$ we obtain
\[
\Phi(u)=F_1(z,w)\cdot F_2(z,w)=u-\frac{1}{\lambda^2} u^3.
\]
The order of $\Phi$ at $u=0$ is $2$. Now the attracting directions
of $\Phi$ at $0$ are $v=\pm \lambda$.
The map $\Phi$ is a polynomial, with  two attracting {\em maximal} petals
$P(\lambda)$ and $P(-\lambda)$ at the origin. The maximal petals
$P(\lambda)$ and $P(-\lambda)$ are disjoint and obtained
as unions of all preimages under $F^{{\circ n}}$, $n=1,2,\ldots$ of
two fixed local petals.

Let $J$ be the Julia set of $\Phi$.
Set $\3J:=\{(z,w): zw\in J\}$. Then $\3J$ has empty interior since $J$ does (see e.g.\ \cite{CaGa}).
We claim that if $(z,w)\notin \3J$, then $F^{\circ n}(z,w)\not\to 0$ as $n\to \infty$.

It is well-known
that, if $u_{0}\notin J$ and $\Phi^{\circ n}(u_{0})\to 0$ as $n\to\infty$,
then $u_{0}\in P(\l)\cup P(-\l)$. Therefore if $(z,w)\not\in\3J$ is such that $zw\not\in P(\l)\cup P(-\l)$ then $\{\Phi^{\circ n}(zw)\}$ cannot converge to $0$ and therefore $F^{\circ n}(z,w)\not\to 0$.

Now, let $(z,w)\not\in\3J$ be such that $zw\in P(\l)\cup P(-\l)$. Then, setting $u_{l}=\Phi^{\circ l}(zw)$ we have
\[
F^{\circ m}(z,w)=\left(\lambda^m z \prod_{l=1}^m \left(1+\frac{u_l}{\lambda} \right), \frac{w}{\l^m}
\prod_{l=1}^m \left(1-\frac{u_l}{\lambda} \right)\right),
\]
so that the behavior of $F^{\circ m}(z,w)$ depends only on the
behavior of the infinite products $\prod_{l=1}^m \left(
1\pm\frac{u_l}{\l}\right)$.

The sequence $\{u_l\}$ tends
to $0$ with speed
\begin{equation}\label{via}
|u_l|^2\sim \frac{1}{l},
\end{equation}
while $u_l/|u_l|\to \pm \l$ depending on whether $zw\in P({\pm \l})$.
But
\[
\left|1\pm \frac{u_l}{\l} \right|= \sqrt{1+|u_l|^2\pm 2|u_l|\Re \frac{u_l}{\l |u_l|}}\sim 1\pm \frac{1}{\sqrt{l}}\Re \frac{v}{\l},
\]
where $v=\pm \l$. Therefore, if $zw\in P(\l)$, i.e. $v=\l$,
then the behavior of $p_1 \circ F^{\circ m}(z,w)$ (here
$p_1(z,w)=z$) depends on the infinite product
\[
\prod \left(1+\frac{1}{\sqrt{l}}\right),
\]
which diverges to $\infty$, while the behavior of $p_2 \circ
F^{\circ m}(z,w)$  (here
$p_2(z,w)=w$)  depends on the infinite product
\[
\prod \left(1-\frac{1}{\sqrt{l}}\right),
\]
which converges to $0$. If $zw\in P(-\l)$ the situation is reversed. In both cases  $F^{\circ n}(z,w)\not\to 0$. Hence $F$ has no basin of attraction at $0$.

\subsection{Example of a one-resonant non-degenerate (but not parabolically attracting) germ
with no basins of attraction}\label{non-attracting}

Consider the germ given by
$$
F(z,w)=(
z-z^{2}, \l w + \l zw
),
$$
where $|\l|=1$ and $\lambda$  is not a root of unity.
Then $F$ is one-resonant with index of resonance $(1,0)$.
Furthermore $\L=-1$, hence $F$ is non-degenerate.
On the other hand, $F$ is not parabolically attracting,
in fact $\Re(a_{2}\l_{2}^{-1}\L^{-1})=-1<0$.
Thus Theorem~\ref{main} does not apply and, in fact, $F$ has no basin of attraction.

Indeed, if $F^{\circ n}(z_{0},w_{0})\to0$ as $n\to\infty$, then $z_{0}$ must belong to the maximal petal
of the map $\phi(z)=z-z^{2}$. Setting $z_{n}:=\phi^{\circ n}(z_{0})$, we have
$$
F^{\circ n}(z_{0},w_{0}) =
\left(
z_{n},
\l^{n} w  \prod_{l=1}^{n} \left(1+z_{l}\right).
\right)
$$
In view of Theorem~\ref{LF},
$$\left|\prod_{l=1}^{n} \left(1+z_{l}\right)\right| \ge \prod_{l=1}^{n} \left(1+\frac{\eps}{l}\right)=+\infty$$
for suitable $\eps>0$. Hence the only possibility for $F^{\circ n}(z_{0},w_{0})\to0$
is when $w_{0}=0$.  Thus we cannot have a (open) basin of attraction.

\bibliographystyle{alpha}

\begin{thebibliography}{ABT2}
\bibitem{Ab} M.  Abate, {\sl The residual index and the
dynamics of holomorphic maps tangent to the identity}.  Duke
Math. J. 107, 1, (2001), 173-207.
\bibitem{Ab1} M. Abate, {\sl Discrete holomorphic local dynamical systems}.
Holomorphic dynamical systems, 1–55, Lecture Notes in Math., 1998, Springer, Berlin, 2010
\bibitem{ABT} M. Abate, F. Bracci, F. Tovena, {\sl Index theorems for
holomorphic self-maps}.  Ann. of Math. 159, 2, (2004), 819-864.
\bibitem{Ar} V. I. Arnold, {\sl Geometrical methods in the theory of
ordinary differential equations}. Springer, 1983.
\bibitem{Br} F. Bracci, {\sl Local dynamics of holomorphic diffeomorphisms}. Boll. UMI (8) 7-B (2004), 609-636.
\bibitem{Br2} F. Bracci, {\sl Local holomorphic dynamics of diffeomorphisms in dimension one}.
Five lectures in complex analysis, 1–42, Contemp. Math., 525, Amer. Math. Soc., Providence, RI, 2010
\bibitem{B-M} F. Bracci\and  L. Molino, {\sl The dynamics near
quasi-parabolic fixed points of holomorphic diffeomorphisms in
$\C^2$}. Amer. J. Math. 126 (2004), 671-686.
\bibitem{BCL} F. Brochero Martinez, F. Cano, L. Lopez Hernanz, {\sl Parabolic
curves for diffeomorphisms in $\mathbb C^2$}. Publ. Math. 52
(2008), 189-194.
\bibitem{CS} C. Camacho, P. Sad, {\sl Invariant varieties through singularities of holomorphic vector fields}. Ann. of
Math. 115 (1982), 579-595.
\bibitem{CaGa} L. Carleson, T. W. Gamelin, {\sl Complex dynamics}.
Springer, 1993.
\bibitem{Ec} J. \'Ecalle, {\sl Les fonctions r\'esurgentes, Tome
III: L'\'equation du pont et la classification analytiques des
objects locaux}. Publ. Math. Orsay, 85-5, Universit\'e de
Paris-Sud, Orsay, 1985.
\bibitem{Fa} P. Fatou, {\sl Substitutions analytiques et
equations fonctionelles de deux variables}. Ann. Sc. Ec. Norm.
Sup. (1924), 67-142.
\bibitem{G1} D.M. Grobman, {\sl Homeomorphism of systems of differential equations}. Dokl. Akad. Nauk. USSR
128 (1959), 880-881.
\bibitem{G2} D.M. Grobman, {\sl Topological
classification of neighbourhoods of a singularity in n-space}.
Math. Sbornik 56 (1962), 77-94.

\bibitem{I} F. Ichikawa, {\sl On finite determinacy of formal vector fields.}
Invent. Math. 70 (1982/83), no. 1, 45--52.

\bibitem{IY} Y. Ilyashenko, S. Yakovenko, {\sl Lectures on analytic differential equations}. Graduate Studies in Mathematics, 86. American Mathematical Society, Providence, RI, 2008.
\bibitem{H} M. Hakim, {\sl Attracting domains for
semi-attractive transformations of $\C^p$}. Publ. Math. 38,
(1994), 479-499.
\bibitem{Ha} M. Hakim, {\sl Analytic transformations of $(\C^p,0)$
tangent to the identity}, Duke Math. J. 92 (1998), 403-428.
\bibitem{Har} P. Hartman, {\sl A lemma in the theory of structural stability of differential equations}. Proc. Am.
Math. Soc. 11 (1960), 610-620.
\bibitem{Po} J. P\"oschel, {\sl On invariant manifolds of complex
analytic mappings near fixed points}. Exp. Math. 4 (1986),
97-109.
\bibitem{N} Y. Nishimura, {\sl Automorphisms analytiques admettant
des sous-vari\'et\'es de points fix\'es attractives dans la
direction transversale} J. Math. Kyoto Univ. 23-2 (1983),
289-299.
\bibitem{Ra} J. Raissy, {\sl Torus actions in the normalization problem}, Jour. Geom. Anal., to
appear.
\bibitem{Ri} M. Rivi, {\sl Parabolic manifolds for semi-attractive
holomorphic germs}. Michigan Math. J. 49, 2, (2001), 211-241.
\bibitem{R1} F. Rong, {\sl Quasi-parabolic analytic transformations of
$\mathbb C^n$} J. Math. Anal. Appl. 343, No. 1, (2008), 99-109
.
\bibitem{R2} F. Rong, {\sl Linearization of holomorphic germs with quasi-parabolic fixed
points.} Ergodic Theory Dyn. Syst. 28, No. 3, (2008), 979-986.
\bibitem{St} L. Stolovitch, {\sl Classification analytique de
champs de vecteurs $1$-r\'esonnants de $(\C^n, 0)$}. Asymptotic
Analysis 12 (1996), 91-143.
\bibitem{Tak} F. Takens, {\sl Forced oscillations and bifurcations}, Global analysis of dynamical systems,
Inst. Phys., Bristol, 2001, Reprint from Comm. Math. Inst.
Rijksuniv. Utrecht, No. 3-1974, 1974, pp. 1-61.
\bibitem{Taka} K. Takano, {\sl On the iteration of holomorphic
mappings}. Funkcial. Ekvac. 17 (1974), 107--156.
\bibitem{Ue1} T. Ueda, {\sl Local structure of analytic
transformations of two complex variables, I}. J. Math. Kyoto
Univ., 26-2 (1986), 233-261.
\bibitem{Ue2} T. Ueda, {\sl Local structure of analytic
transformations of two complex variables, II}. J. Math. Kyoto
Univ., 31-3 (1991), 695-711.
\bibitem{Y} J.-C. Yoccoz: {\sl Th\'eor\'eme de Siegel, nombres de Bryuno et polynomes quadratiques}. Ast\'erisque 231
(1995), 3–88.
\bibitem{We} B. J. Weickert, {\sl Attracting basins for
automorphisms of $\C^2$}. Invent. Math. 132 (1998), 581-605.
\end{thebibliography}

\end{document}